\documentclass[reqno]{amsart}
\usepackage{epsfig}
\usepackage[utf8]{inputenc} 
\usepackage[all]{xy}
\usepackage{amssymb}
\usepackage{amsmath}
\newtheorem{theorem}{Theorem}
\newtheorem{lemma}[theorem]{Lemma}
\newtheorem{definition}[theorem]{Definition}
\newtheorem*{remark}{Remark}{\it}{\rm}

\newtheorem{proposition}[theorem]{Proposition}
\newtheorem*{bem}{Remark}{\it}{}

\numberwithin{equation}{section}

\newcommand{\Z}{{\mathbb Z}}

\newcommand{\lcm}{{\rm lcm}}

\newcommand{\R}{{\mathbb R}}

\newcommand{\C}{{\mathbb C}}

\newcommand{\tr}{{\rm tr}}

\begin{document}

\title[Petersson products]
{Petersson products of bases of spaces of cusp forms and estimates for
Fourier coefficients}  
\author[R. Schulze-Pillot]{Rainer Schulze-Pillot} 
\author[A. Yenirce]{Abdullah Yenirce}
 \begin{abstract}
We prove a bound for the Fourier coefficients of a cusp form of
integral weight which is not a newform by computing an explicit
orthogonal basis for the space of cusp forms of given integral weight and level.
\end{abstract}

\maketitle

\section{Introduction.} 
The Fourier coefficients $a(F,n)$ of a cusp form $F$ of integral weight $k$ for
the group $\Gamma_0(M)$ are bounded from above by
$\sigma_0(n)n^\frac{k-1}{2}$ if $F$ is a primitive form (also called
normalized newform) by the famous Ramanujan-Petersson-Deligne
bound. For applications one often needs bounds for an arbitrary cusp
form which is a linear combination of old and new forms. Such bounds
have first been given in special cases in \cite{fomenko1,fomenko2}.
The first step for this is the construction of an explicit
orthogonal basis for the space $S_k(\Gamma_0(M), \chi)$. 
Starting from
the usual basis of translates of primitive forms and using the well
known fact that translates of
different primitive forms are pairwise orthogonal, one is left with
the task to orthogonalise the translates of the same primitive
form, in particular, one has to compute their Petersson scalar products.

Choie and Kohnen in \cite{choie-kohnen} and Iwaniec, Luo, Sarnak in
\cite{iwa_luo_sar} cover arbitrary integral weights, square free
level and  trivial character, using Rankin $L$-functions for the
computation of the Petersson products of translates of a primitive form. By the same
method, Rouymi \cite{rouymi} treated prime power level and trivial
character. His approach was generalized to arbitrary levels and
trivial character by Ng Ming Ho in his unpublished master thesis 
\cite{ngmingho}. Blomer and Mili\'{c}i\'{c} in \cite{blomer_milicic} treat
Maa\ss forms and holomorphic modular forms for arbitrary level and
trivial character by the same method.

In this note we investigate the case of arbitrary level and arbitrary
character  with a rather elementary 
approach. 
In order to compute the Petersson product of two translates of the
same primitive form we use in Section 2 the trace operator sending a form of level $M$ to a
form of level $N$ dividing $M$. Together with the well known fact that the $p$-th Hecke
operator on forms of level $N$ can be obtained by first translating the argument by a factor
$p$ and then applying the trace operator from level $Np$ down to level
$N$ this allows us to express the scalar products quite easily in
terms of Hecke eigenvalues of the underlying primitive form.

We can then in Section 3 use the Gram-Schmidt procedure to compute an orthogonal
basis of the space $S_k(\Gamma_0(M),\chi)$ of all cusp forms of some
given level $M$ and character $\chi$. The formulas for the Petersson
products we obtained and the relations between
the Hecke eigenvalues $\lambda_f(1,p^j)$ of a primitive form $f$ for
varying $j$ imply then that each element of this orthogonal basis
involves
only very few of the translates of its underlying primitive form, so
that is easy to compute the Petersson norms of these basis elements
and to estimate their Fourier coefficients.

For
forms of half integral weight we show in Section 4 that  our approach
works in essentially the 
same way as far as the computation of the Petersson product of a Hecke
eigenform with its translates  is concerned. However,  since the theory of
newforms of half integral weight is  completely known only for the Kohnen plus
space in square free level, we do for general level not know how large
the proportion
of the space of all cusp forms is
in which we can estimate the Fourier coefficients.

We then use in Section 5 in the integral weight case the orthogonal
basis to obtain  as our main result in Theorem \ref{fourier_estimate} 
an explicit bound for the Fourier coefficient $a(F,n)$ of an arbitrary
cusp form $F$ in terms of the Petersson norm $\langle
F,F \rangle$ and the level $M$.

In applications to the theory of integral quadratic forms it is
usually possible to compute or at least bound $\langle F,F \rangle$ for the cusp form
$F$ at hand (the difference between a genus theta series and a theta
series), so that our result is directly applicable to such problems;
this will be worked out separately.

An estimate for the Fourier coefficients in the half integral weight
case could in principle be obtained in the same way as in the integral
weight case discussed above as long as one has an explicit bound for
the Fourier coefficients with square free index of a Hecke
eigenform. Unfortunately most of the known estimates (see
\cite[Appendix 2]{blomer_michel_mao} involve
constants which are not explicitly known, and we prefer not to discuss
this possibility in detail in the present paper.

This article is an extension of work from the master thesis of the
second named author at Universität des Saarlandes, 2014.

After the first version of this article was posted in the matharxiv, Ng
Ming Ho sent us his master thesis \cite{ngmingho}, from which we also learnt of the
previous work of Iwaniec, Luo and Sarnak, and of  Rouymi.   
We thank Ng Ming Ho for providing this information to us. We also
thank the referee for his very detailed comments which eliminated
several embarrassing typing errors.
 \section{Trace operator and scalar products} 
Let $N\mid M$ be integers and let $\chi$ be a Dirichlet character modulo $N$; we denote the Dirichlet
character modulo $M$ induced by it by $\chi$ as well. We have induced characters on the groups 
$\Gamma_0(N)$, $\Gamma_0(M)$ given by
 \begin{equation*} 
\begin{pmatrix} a & b \\ c & d \end{pmatrix} \longmapsto \chi(d) 
 \end{equation*}
as usual and denote these again by $\chi$.

For an integer $k$ we denote by
$M_k(\Gamma_0(N),\chi),S_k(\Gamma_0(N),\chi)$ the spaces of modular
forms respectively cusp forms of weight $k$ and character $\chi$ for
the group $\Gamma_0(N)$. On $S_k(\Gamma_0(N),\chi)$ we consider the
Petersson inner product given by
\begin{equation*}
\langle f,g\rangle:=\langle f,g\rangle_{\Gamma_0(N)}:=\frac{1}{(SL_2(\Z):\Gamma_0(N))}\int_{\mathcal F}f(x+iy)\overline{g(x+iy)}y^{k-2}dxdy, 
\end{equation*}
where ${\mathcal F}$ is a fundamental domain for the action of
$\Gamma_0(N)$ on the upper half plane $H\subseteq \C$ by fractional
linear transformations. The normalization chosen implies that for
$N\mid M$ and $f,g \in S_k(\Gamma_0(N),\chi)\subseteq
S_k(\Gamma_0(M),\chi)$ we have  $\langle f,g
\rangle_{\Gamma_0(N)}=\langle f,g \rangle_{\Gamma_0(M)}$.

For $\gamma=\bigl(
\begin{smallmatrix}
  a&b\\c&d
\end{smallmatrix}\bigr)
\in GL_2(\R)$ with $\det(\gamma)>0$ we write as usual
$f\vert_k\gamma(z)=\det(\gamma)^{k/2}(cz+d)^{-k}f(\frac{az+b}{cz+d})$.
 
\medskip
We define trace operators as in \cite{kume,bfsp}: 

 \begin{definition}
For $N\mid M$ and $\chi$ as above we put for $f \in M_k(\Gamma_0(M),\chi)$
\begin{equation}\label{trace_definition}  
 f|_k {\rm tr}_N^M = \frac{1}{(\Gamma_0(N):\Gamma_0(M))} \sum_{i} \overline{\chi(\alpha_i)} f\vert_k \alpha_i,
 \end{equation}
 where $\displaystyle \Gamma_0(N) = 
 \dot{\bigcup_i}\,\Gamma_0(M) \alpha_i$ is a disjoint coset
decomposition.
 \end{definition}

 \begin{lemma}
The definition above is independent of the choice of coset representatives. 
One has $f|_k {\rm tr}_N^M \in M_k(\Gamma_0(N),\chi)$ and $f|_k {\rm tr}_N^M \in S_k(\Gamma_0(N),\chi)$ if $f$ is
cuspidal.
 \end{lemma}

 \begin{proof} 
This is a routine calculation, see e.g. \cite[Proposition 2.1]{bfsp}.
 \end{proof}

 \begin{lemma} \label{trace-skp}
With notations as above one has for $f \in S_k(\Gamma_0(N),\chi)$, $g \in S_k(\Gamma_0(M),\chi)$:
 \begin{equation*} 
 \langle f,g \rangle = \langle f,g~|_k ~{\rm tr}_N^M \rangle,
 \end{equation*}
where the Petersson product on the left hand side is with respect to
$\Gamma_0(M)$ and that on the right hand side is with respect to $\Gamma_0(N)$.
 \end{lemma}

 \begin{proof}
One has  $\langle f,\overline{\chi}(\alpha_i)g|_k \alpha_i \rangle 
 = \langle \chi(\alpha_i)f|_k \alpha_i^{-1}),g \rangle
= \langle f,g\rangle$,
which implies the assertion.
 \end{proof}

\newpage
 \begin{definition}
 Let $\gcd(\ell,N) = 1$.
 \begin{itemize} 
 \item[a)] With $\delta_{\ell} := \begin{pmatrix} \ell & 0 \\ 0 & 1 \end{pmatrix} \in {\rm GL}_2^+(\mathbb Q)$ we put
 \begin{equation*} 
 f|_k V_{\ell}(z) := f(\ell z) = \ell^{-k/2} f|_k \delta_{\ell} (z)
 \end{equation*}
for $f \in M_k(\Gamma_0(N),\chi)$.
 \item[b)] For $\ell \mid m$ we denote by $T_N(\ell,m)$ the Hecke operator given by the double coset $\Gamma_0(N) \begin{pmatrix} \ell & 0 \\ 0 & m\end{pmatrix} \Gamma_0(N)$.
 \item[c)] For $\ell\mid m$ we denote by $T^{\ast}_N(m,\ell)$ the Hecke operator given by the double coset 
$\Gamma_0(N) \begin{pmatrix} m & 0\\ 0 & \ell \end{pmatrix} \Gamma_0(N)$.
 \end{itemize}
 \end{definition}

 \begin{remark}
 \begin{itemize}
 \item[a)] It is well-known (see \cite[\S 4.5]{Miy}) that on spaces of cusp forms of level $N$ the
operator $T^{\ast}(m,\ell)$ is adjoint to $T(m,\ell)$ with respect to the Petersson inner product.
 \item[b)] As usual we write
 \begin{equation*} 
 T_N(n) = \sum_{\ell m = n}T_N(\ell,m),\: T_N^{\ast}(n) = \sum_{\ell m = n} T_N^{\ast} (m,\ell).
 \end{equation*}
 \end{itemize}
 \end{remark}

 \begin{lemma} \label{trace-hecke} 
Let $f \in S_k(\Gamma_0(N),\chi)$ and $d \in \mathbb N$. Then
 \begin{equation} \label{trace_hecke_formula}
  (\Gamma_0(N):\Gamma_0(Nd))(f|_k V_d)~|_k~{\rm tr}_N^{Nd} = \frac{1}{d^{k-1}} f|_k T_N^{\ast}(d,1).
 \end{equation}
 \end{lemma}

 \begin{proof}
Putting $\Gamma_0(N) = {\underset{i}{\stackrel{\cdot}{\bigcup}}} \Gamma_0(Nd)\alpha_i$ we have 
(using $\delta_d^{-1} \Gamma_0(N){\delta_d} \cap \Gamma_0(N) =
\Gamma_0(Nd)$ and the proof of Proposition 3.1 of \cite{shimura_book}):
 \begin{align*}
 f|_k T_N^{\ast}(d,1)& = d^{k-1} \sum_i \overline{\chi(\alpha_i)} (d^{-k/2}f|_k \delta_d)|_k \alpha_i\\
 & = d^{k-1}\sum_i \overline{\chi(\alpha_i)} (f|_k V_d)|_k \alpha_i\\
 &=(\Gamma_0(N):\Gamma_0(Nd)) d^{k-1}(f|_k V_d)|_k {\rm tr}_N^{Nd}.
 \end{align*}
 \end{proof}

 \begin{theorem}\label{gram_matrix_theorem}
Let $f \in S_k(\Gamma_0(N),\chi)$ be a primitive form, let $m,n \in \mathbb N$ with $\gcd(m,n) = d$. 
Then
 \begin{equation}\label{gram_matrix_formula}
\langle f|_k V_m,f|_kV_n\rangle  = \frac{\lambda(1,\frac{n}{d})\overline{\lambda(1,\frac{m}{d})}}
{ (\frac{mn}{d})^k \underset{p|\frac{mn}{d^2} \atop p\nmid N}{\prod} (1+\frac{1}{p})} \langle f,f \rangle ,
 \end{equation}
where we denote by $\lambda(1,\frac{n}{d})$ the $T(1,\frac{n}{d})$-eigenvalue of $f$ (and analogously
for $\lambda(1,\frac{m}{d})$).
 \end{theorem}

 \begin{proof}
 Since we have
 \begin{align*}
 \langle f|_k V_m, f|_k V_n \rangle &= \langle f|_k V_{m/d} |_k V_d,f|_k V_{n/d}|_k V_d \rangle\\
 &=d^{-k} \langle f|_k V_{m/d}|_k \delta_d,f|_k V_{n/d}|_k \delta_d \rangle \\
 &= d^{-k} \langle f|_k V_{m/d},f|_k V_{n/d}\rangle,
 \end{align*}
we can restrict attention to the case
 \begin{equation*}
d= \gcd(m,n) = 1.
 \end{equation*}
In that case we have 
\begin{align*}
 & \langle f|_k V_m,f|_k V_n \rangle = \langle f|_k V_m,f|_k V_n ~|~ {\rm tr}_{mN}^{mnN} \rangle\\
&= \frac{1}{(\Gamma_0(mN):\Gamma_0(mnN))} \frac{1}{n^{k-1}} \langle f|_k V_m, f|_k T_{mN}^{\ast} (n,1) \rangle,
 \end{align*}
where we used Lemma \ref{trace-skp} and Lemma \ref{trace-hecke}.
 \medskip
We split $n$ as $n=\tilde{n}n'$ with $\gcd(\tilde{n},N)=1$ and $n'|N^{\infty}$ (i.e., $n'$ is divisible 
only by primes dividing $N$) and have 
\begin{align*}
 T_{mN}^{\ast} (n,1) &= T_{mN}^{\ast}(\tilde{n},1) T_{mN}^{\ast}(n',1),\\
f|_k T_{mN}^{\ast}(\tilde{n},1) &= f|_kT_N^{\ast}(\tilde{n},1)\\
&= \overline{\lambda(1,\tilde{n})} f,
 \end{align*}
since $T_N^{\ast}(\tilde{n},1)$ is adjoint to
$T_N^{\ast}(1,\tilde{n})$.

In the same way we see 
 \begin{align*}
 f|_k T_{mN}^{\ast} (n',1)&= f|_k T_N^{\ast}(n',1)\\
 &=\overline{\lambda(1,n')}f.
 \end{align*}
This gives us
 \begin{align*}
 \langle f|_k V_m,f|_k V_n\rangle &= \frac{1}{n^{k-1}(\Gamma_0(mN): \Gamma_0(mnN))} \cdot
   \lambda(1,\tilde{n}) \lambda (1,n') \langle f|_k V_m,f \rangle\\
 &= \frac{\lambda(1,n)}{n^{k-1}(\Gamma_0(mN):\Gamma_0(mnN))} \overline{\langle f,f|_k V_m \rangle}.
 \end{align*}
In particular, we get
 \begin{equation*}
 \langle f,f|_k V_m\rangle = \frac{\lambda(1,m)}{(\Gamma_0(N):\Gamma_0(mN))m^{k-1}}
  \overline{\langle f,f \rangle},
 \end{equation*}
and thus (computing the group index in the denominator )
 \begin{align*}
 \langle f|_k V_m,f|_k V_n\rangle &= \frac{\lambda(1,n)\overline{\lambda(1,m)}}{(mn)^{k-1}(\Gamma_0(N):\Gamma_0(mN))}
 \langle f,f \rangle\\
 &= \frac{\lambda(1,n) \overline{\lambda(1,m)}}{(mn)^k \underset{p|mn
   \atop p\nmid N}{\prod} (1+\frac{1}{p})}
  \langle f,f \rangle
 \end{align*}
as asserted.
 \end{proof}

  \section{Orthogonal bases for spaces of cusp forms}
The formulas for the Petersson products derived in the previous
section allow to construct an orthogonal basis by Gram Schmidt
orthogonalization. As we learnt from Ng Ming Ho after version one of
this article was posted, this has been done for trivial character in
\cite{rouymi} for prime power level 
and in \cite{ngmingho} for general level. For the sake of completeness
and since \cite{ngmingho} is at present not published we give here our
version of it.

\smallskip
We recall first the well-known fact (see e.g. \cite[Lemma 4.6.9]{Miy} that the space $S_k(\Gamma_0(M),\chi)$ 
has a basis consisting of the $f|_{V_{\ell}}$, where $f$ runs over the primitive forms (normalized Hecke
eigenforms) of levels $N\mid M$ where $N$ is divisible by the conductor of $\chi$, and where $\ell$ is
a positive integer such that $\ell N$ divides $M$. We will call this basis the basis of translates of
newforms.

 \begin{lemma}\label{product_decomposition-Lemma}
Let $f\in S_k(\Gamma_0(N,\chi))$ be a primitive form, let $m_1,m'_1,m_2,m'_2$ be positive integers with
$ \gcd(m_1m'_1,m_2m'_2)=1$, and
put $\tilde{f} = \frac{f}{\sqrt{\langle f,f \rangle}}$. Then
 \begin{align*}
  \langle \tilde{f}|_k V_{m_1},\tilde{f}|_k V_{m'_1} \rangle \cdot \langle \tilde{f}|_k V_{m_2},\tilde{f}|_k V_{m'_2} 
  \rangle \\
  = \langle \tilde{f}|_k V_{m_1m_2},\tilde{f}|_k V_{m'_1m'_2} \rangle\,.
 \end{align*}
 \end{lemma}

 \begin{proof}
This follows directly from the theorem above.
 \medskip
\end{proof}
It is well-known that for primitive forms $f \not= g$ all translates of $f$ by some $V_{m'}$
are orthogonal to all translates of $g$ by some $V_{m'}$. Our lemma above shows that for a primitive form
$f \in S_k(\Gamma_0(N),\chi)$ for some $N\mid M$ the space of translates of $f$ in $S_k(\Gamma_0(M),\chi)$ is 
isometric (with respect to Petersson norms) to the tensor products of the spaces $W_{p_i}{(f)}$ for the
$p_i|\frac{M}{N}$ consisting of $p_i$-power-translates of $f$. An isometry is given by the unique linear
map with
 \begin{equation*} 
 \tilde{f}|_k{V_{p_1^{r_1}}} \otimes \dots \otimes \tilde{f}|_k{V_{p_t^{r_t}}} \mapsto
 \tilde{f}|_k{V_{p_1^{r_1} \dots p_t^{n_t}}},
 \end{equation*}
where $\displaystyle \tilde{f} = \frac{f}{\sqrt{\langle f,f \rangle}}$.
 \medskip

To construct an orthogonal basis for  $S_k(\Gamma_0(M),\chi)$ it suffices therefore to do that for
each space $W_{p_i}(f)$.

 \begin{theorem}\label{ogbasis_prime}
Let $f \in S_k(\Gamma_0(N),\chi)$ be a primitive form, put  $\tilde{f}=\frac{f}{\sqrt{\langle f,f \rangle}}$,
let $p$ be a prime number, $r \in \mathbb N$, let $W_{p,r}(f)$ be the space generated by
$f,f|_k{V_p},\ldots,f|_k{V_{p^r}}$.
 \medskip

 \begin{itemize}
 \item[a)] If $p\mid N$ the space $W_{p,r}(f)$ has an orthogonal basis consisting of
   \begin{equation}\label{basisvectors1}
      g_0 = \tilde{f}, \quad g_j = p^{jk/2}(\tilde{f}|_k{V_{p^j}} - \frac{\overline{\lambda(1,p)}}{p^k} 
 \tilde{f}|_k{V_{p^{j-1}}} )
 \text{ for } 1 \leq j \leq r
   \end{equation}
with 
\begin{eqnarray*}
\langle g_0,g_0 \rangle& =& 1,\\
\langle g_j,g_j \rangle& =&1-\frac{|\lambda(1,p)|^2}{p^k}
\text{ for }1 \leq j \leq r.
\end{eqnarray*}
 \item[b)] If $p\nmid N$ the space $W_{p,r}(f)$ has an orthogonal basis consisting of
   \begin{eqnarray}\label{basisvectors2}
     g_0 &=& \tilde{f},\nonumber\\
     g_1&=& p^{k/2} \tilde{f}|_k V_p -  \frac{\overline{\lambda(1,p)}}
{p^{k/2}(1+\frac{1}{p})} \tilde{f},\nonumber\\
g_j &=& p^{jk/2}(\tilde{f}|_k V_{p^j} - \frac{\overline{\lambda(1,p)}}{p^k}
    \tilde{f}|_k V_{p^{j-1}}+ \frac{\overline{\chi(p)}}{p^{k+1}}
        \tilde{f} |_k V_{p^{j-2}}) 
   \end{eqnarray}
for $2 \leq j \leq r$, with 
\begin{eqnarray*}
\langle g_0,g_0 \rangle &=&1,\\
\langle g_1,g_1 \rangle &=&  1-\frac{|\lambda(1,p)|^2}
 {p^k(1+\frac{1}{p})^2},\\
\langle g_j,g_j \rangle &=& (1-\frac{1}{p^2})(1- \frac{|\lambda(1,p)|^2}{p^k(1+\frac{1}{p})^2})
\quad (2 \leq j \leq r). 
\end{eqnarray*}
 \end{itemize}
 \end{theorem}
 
 \begin{proof}
a) In the case $p|N$ we have by Theorem \ref{gram_matrix_theorem}
for $0 \leq i \leq j \leq r$:
 \begin{align*}
 \langle \tilde{f}|_k V_{p^i}, \tilde{f}|_k V_{p^j} \rangle &= p^{-ik} \langle \tilde{f},\tilde{f}|_k V_{p^{j-i}} \rangle\\
 &= p^{-ik} \frac{\lambda(1,p^{j-i})}{p^{(j-i)k}}\\
 &=p^{-jk}\lambda(1,p^{j-i})
 \end{align*}
This gives for $1 \leq j \leq r$
 \begin{align*}
 \langle g_0,g_j \rangle &= p^{jk/2} \langle \tilde{f},\tilde{f}|_k V_{p^j} - \frac{\lambda(1,p)}{p^k}
  \tilde{f}|_k V_{p^{j-1}}\rangle\\
&= p^{jk/2} (\frac{{\lambda(1,p^j)}}{p^{jk}} - \frac{{\lambda(1,p)} {\lambda(1,p^{j-1})}}
 {p^k p^{(j-1)k}})\\
 &= 0\, ,
\end{align*}
because of $\lambda(1,p) \lambda(1,p^{j-1}) = \lambda(1,p^j)$ for $p|M$.
 \medskip

Similarly, we see for $1 \leq i < j \leq r$ 
 \begin{align*}
 \langle g_i,g_j \rangle &= p^{(i+j)k/2} \langle \tilde{f}|_k V_{p^i} - \frac{\overline{\lambda(1,p)}}{p^k} 
 \tilde{f}|_k V_{p^{i-1}},\tilde{f}|_k V_{p^j}-\frac{\overline{\lambda(1,p)}}{p^k} \tilde{f}|_k V_{p^{j-1}} \rangle\\
 =& p^{(i+j)k/2}(p^{-jk} \lambda(1,p^{j-i})+ \frac{|\lambda(1,p)|^2}{p^{2k}} p^{-(j-1)k} \lambda(1,p^{j-i})\\
 &- \frac{\overline{\lambda(1,p)}}{p^k} p^{-jk} \lambda(1,p^{j-i+1})-
   \frac{\lambda(1,p)}{p^k} p^{-(j-1)k} \lambda(1,p^{j-i-1}))\\ 
  =& 0
 \end{align*}
because of $\lambda(1,p) \lambda(1,p^{j-i-1}) = \lambda(1,p^{j-i})$ and
$\overline{\lambda(1,p)} \lambda(1,p^{j-1+1}) = \overline{\lambda(1,p)} \lambda(1,p) \lambda(1,p^{j-i})
= |\lambda(1,p)|^2 \lambda(1,p^{j-i})$.
 \medskip

Finally we have for $1 \leq j \leq r$
 \begin{align*}
 \langle g_j,g_j \rangle& = p^{jk} \langle \tilde{f}|_k V_{p^j}-\frac{\overline{\lambda(1,p)}}{p^k}
  \tilde{f} |_kV_{p^{j-1}}, \tilde{f}|_k V_{p^j} - \frac{\lambda(1,p)}{p^k} \tilde{f}|_k V_{p^{j-1}} \rangle\\
 & = p^{jk}(p^{-jk}+ \frac{p^{-(j-1)k}}{p^{2k}} |\lambda(1,p)|^2-\frac{\overline{\lambda(1,p)} p^{-jk}}{p^k}
  \lambda(1,p) -\frac{\lambda(1,p)}{p^k} p^{-jk} \overline{\lambda(1,p)}) \\
 & = (1- \frac{|\lambda(1,p)|^2}{p^k})
 \end{align*}
b) Consider now the case $p \nmid N$. From Theorem \ref{gram_matrix_theorem} we have for $0 \leq i < j \leq r$
 \begin{equation*}
 \langle \tilde{f}|_k V_{p^i},\tilde{f}|_k V_{p^j} \rangle = \frac{\lambda(1,p^{j-i})}{p^{jk}(1+\frac{1}{p})}
\end{equation*}
and 
 \begin{equation*}
\langle \tilde{f}|_k V_{p^j}, \tilde{f}|_k V_{p^j}\rangle = \frac{1}{p^{jk}}.
 \end{equation*}
From \cite[Lemma 4.5.7]{Miy} we have $\lambda(1,p^2) = \lambda(1,p)^2-(p+1)p^{k-2}\chi(p)$ and 
$\lambda(1,p^j) = \lambda(1,p)\lambda(1,p^{j-1})-p^{k-1}\chi(p) \lambda(1,p^{j-2})$ for $j \geq 3$.
\medskip

This gives us first
 \begin{align*}
 \langle g_0,g_1 \rangle &= p^{k/2} \langle \tilde{f},\tilde{f}|_k V_p \rangle - \frac{\lambda(1,p)}
 {p^{k/2}(1+\frac{1}{p})}\\
&=0\,.
 \end{align*}
For $i \geq 1$ we get
 \begin{align*}
p^{-(i+1)k/2} \langle \tilde{f}|_k V_{p^i},g_{i+1} \rangle  =& \langle \tilde{f}|_k V_{p^i},\tilde{f}|_k V_{p^{i+1}} \rangle
 - \frac{\lambda(1,p)} {p^k}\langle \tilde{f}|_k V_{p^i},\tilde{f}|_k V_{p^i} \rangle \\ 
 & \quad +\frac{\chi(p)}{p^{k+1}} \langle \tilde{f}|_k V_{p^i},\tilde{f}|_k V_{p^{i-1}} \rangle \\
 =& \frac{\lambda(1,p)}{p^{(i+1)k}(1+{\frac{1}{p}})} -\frac{\lambda(1,p)}{p^k} \frac{1}{p^{ik}} +
   \frac{\chi(p)}{p^{k+1}} \frac{\overline{\lambda(1,p)}}{p^{ik}(1+\frac{1}{p})}\\
 =& 0
\end{align*}
(using $\chi(p) \overline{\lambda(1,p)}=\lambda(1,p)$, see \cite[Theorem 4.5.4]{Miy}).

\medskip
For $0 \leq i < j \leq r$ with $j \geq 2+i$ we obtain
 \begin{align*}
p^{-jk/2} \langle \tilde{f}|_k V_{p^i},g_j \rangle &= \langle \tilde{f}|_k V_{p^i}, \tilde{f}|_k V_{p^j} \rangle 
 -\frac{\lambda(1,p)}{p^k} \langle \tilde{f}|_k V_{p^i}, \tilde{f}|_k V_{p^{j-1}} \rangle \\
 &\quad \quad+ \frac{\chi(p)}{p^{k+1}} \langle \tilde{f}|_k V_{p^i},\tilde{f}|_k V_{p^{j-2}} \rangle \\
 &= \frac{\lambda(1,p^{j-i})}{p^{jk}(1+\frac{1}{p})} - \frac{\lambda(1,p)\lambda(1,p^{j-i-1})}
  {p^k p^{(j-1)k}(1+\frac{1}{p})}
 + \frac{\chi(p) \lambda(1,p^{j-i-2})}{p^{k+1}p^{(j-2)k}(1+\frac{1}{p})}\\
 &=0 \,.
  \end{align*}
Taken together we see that the $g_i$ form an orthogonal basis, it remains to compute the $\langle g_i,g_i\rangle$.\\
For this, $\langle g_0,g_0\rangle = 1$ is clear.

 \medskip
Next, we have 
 \begin{align*}
\langle g_1,g_1\rangle &= \langle g_1,p^{k/2} \tilde{f}|_k V_p \rangle \\
&= p^k \langle \tilde{f}|_k V_p,\tilde{f}|_k V_p\rangle -p^{k/2} \cdot
  \frac{\overline{\lambda(1,p)}}{p^{k/2}(1+\frac{1}{p})} \langle \tilde{f},\tilde{f}|_k V_p \rangle\\
&= 1-\frac{\overline{\lambda(1,p)}}{(1+\frac{1}{p})} \cdot \frac{\lambda(1,p)}{p^k(1+\frac{1}{p})}\\
&= 1-\frac{|\lambda(1,p)|^2}{p^k(1+\frac{1}{p})^2}.
  \end{align*}

For $j \geq 2 $ we already know that  $g_j$ is orthogonal to the
$f|_kV_{p^i}$ with $i<j$ and see
 \begin{align*}
\langle g_j,g_j \rangle =& \langle g_j,p^{jk/2} \tilde{f}|_k V_{p^j} \rangle \\
 =&p^{jk} \langle \tilde{f}|_k V_{p^j}, \tilde{f}|_k V_{p^j} \rangle 
 - \frac{\overline{\lambda(1,p)}}{p^k} p^{jk} \langle \tilde{f}|_k V_{p^{j-1}},\tilde{f}|_k V_{p^j} \rangle \\
 &+ \quad \frac{\overline{\chi(p)}p^{jk}\langle \tilde{f}|_k V_{p^{j-2}},\tilde{f}|_k V_{p^j} \rangle} {p^{k+1}}\\
 =& 1-\frac{|\lambda(1,p)|^2}{p^k(1+\frac{1}{p})} + \frac{\overline{\chi(p)} \lambda(1,p^2)}
  {p^{k+1}(1+\frac{1}{p})}.
 \end{align*}
Using again $\lambda(1,p^2) = \lambda(1,p)^2-(p+1)p^{k-2}\chi(p)$ and
$\chi(p) \overline{\lambda(1,p)}=\lambda(1,p)$ we obtain the assertion.
\end{proof}
\section{Half integral weights}
For positive integers $\kappa,N$ we denote by $M_{k}(4N, \chi)$
the space of holomorphic modular forms of weight $k=\kappa+\frac{1}{2}$ and
character $\chi$ for the group $\Gamma_0(4N)$. For the relevant
definitions and notations see \cite{shimura_halfintegral}.
In particular, we denote by ${\mathfrak G}$ the covering group of
$GL_2^+(\R)$ defined there and by $\gamma \mapsto \gamma^*$ the
embedding of $\Gamma_0(4)$ into ${\mathfrak G}$ with image
$\Delta_0(4)$. We can extend this embedding by putting $\bigl(
\begin{smallmatrix}
  1&0\\0&m^2
\end{smallmatrix}\bigr)^*=\bigl(\bigl(
\begin{smallmatrix}
  1&0\\0&m^2
\end{smallmatrix}\bigr), m^{\frac{1}{2}}\bigr)$ and 
 $\bigl(
\begin{smallmatrix}
  m^2&0\\0&1
\end{smallmatrix}\bigr)^*=\bigl(\bigl(
\begin{smallmatrix}
  m^2&0\\0&1
\end{smallmatrix}\bigr), m^{-\frac{1}{2}}\bigr)$ and $(\gamma_1 \alpha
  \gamma_2)^*=\gamma_1^* \alpha^* \gamma_2^*$ for $\gamma_1,\gamma_2
  \in \Gamma_0(4)$ and $\alpha$ one of the above matrices of
  determinant $m^2$ . In the sequel we will omit the superscript $*$
  if this can cause no confusion.

We also use the
action of double cosets  of integral matrices of non zero square
determinant  on half integral weight modular forms of level
$4N$ as defined there. In particular we
have associated to the double coset with respect to $\Delta_0(4N)$ of $\bigl(\bigl(
\begin{smallmatrix}
 1&0\\0&m^2
\end{smallmatrix}\bigr),m^{\frac{1}{2}} \bigr)$
the Hecke
operators $T_{4N}(1,m^2)$ which for  $m\mid 4N$ coincide with the the operators $U(m^2) $
sending $\sum_n a_f(n)e(nz)$  to $\sum_n a_f(nm^2)e(nz)$. By
considering a modular form of level $4N$ as a form of level
$\lcm(m,4N)$ we can let $U(m^2)$ act on forms of any level divisible
by $4$. 
The operator $T^*_{4N}(m^2,1)$ associated to the double coset of  $\bigl(
\begin{smallmatrix}
  m^2&0\\0&1
\end{smallmatrix}\bigr)^*$ is adjoint to $T_{4N}(1,m^2)$ with respect
to the Petersson product and coincides with it if one has
$\gcd(m,4N)=1$, we write then as usual  $T_{4N}(m^2)$.
For $N$ dividing $M$ we have as in the integral weight case a trace
operator $\tr^M_N$ from  $M_{k}(4M, \chi)$ to $M_{k}(4N, \chi)$
sending cusp forms to cusp forms and satisfying for cusp forms $f,g$
\begin{equation*} 
 \langle f,g \rangle = \langle f,g~|_k ~{\rm tr}_N^M \rangle,
 \end{equation*}
where the Petersson product on the left hand side is with respect to
$\Gamma_0(M)$ and that on the right hand side is with respect to $\Gamma_0(N)$.

In the theory of half integral weight modular forms there are two
different methods used for the definition of oldforms,  namely using
the operator $V_{d^2}$ as in the integral weight case (but with square
determinant), raising the level by a factor  $d^2$, and using the
operator $U(p^2)$ for a prime not dividing the level, raising the
level by a factor  $p$.
We start with the first method.
\begin{proposition}
\label{trace-hecke_halfintegral} 
Let $k=\kappa+\frac{1}{2}$ be half integral,
let $f \in S_k(\Gamma_0(N),\chi)$ and $d \in \mathbb N$. Then
 \begin{equation*} 
  (\Gamma_0(N):\Gamma_0(Nd^2))(f|_k V_{d^2})~|_k~{\rm tr}_N^{Nd^2} = \frac{1}{d^{2(k-1)}} f|_k T_N^{\ast}(d^2,1).
 \end{equation*}  
In particular, if $p$ is a prime with $p\nmid 4N$ and $f$ is an
eigenform of the Hecke operator $T(p^2)$ with eigenvalue $\lambda_p$,
we have 
 \begin{equation*} 
 (p^2+p)(f|_k V_{p^2})~|_k~{\rm tr}_N^{Np^2} =
  \frac{ \lambda_p}{p^{2(k-1)}} f
 \end{equation*}  
and 
\begin{equation*}
  \langle f, f|_k V_{p^2} \rangle =\frac{
    \lambda_p}{(p^2+p)p^{2(k-1)}}\langle f,f\rangle.
\end{equation*}

\end{proposition}
\begin{proof}
  This is proven in the same way as Lemma \ref{trace-hecke}. Notice
  that in the case of half integral weight we can only use shift
  operators $V_{d^2}$ and Hecke operators $T_N^{\ast}(d^2,1)$ with
  squares $d^2$.
\end{proof}
\begin{proposition}
\label{trace-hecke_halfintegral} 
Let $k=\kappa+\frac{1}{2}$ be half integral,
let $f \in S_k(\Gamma_0(N),\chi)$ and $p\nmid 4N$ be a prime.

Then 
\begin{equation*}
  f \mid_k U(p^2)|_k \tr^{Np}_N=p^2 f|_kT(p^2).
\end{equation*}
In particular, if $f$ is an
eigenform of the Hecke operator $T(p^2)$ with eigenvalue $\lambda_p$,
we have 
 \begin{equation*}
\langle f,f|U(p^2) \rangle =p^2 \lambda_p \langle f, f\rangle.
 \end{equation*}  
\end{proposition}
\begin{proof}
 With $\alpha_b=\bigl(
 \begin{smallmatrix}
   1&b\\0&p^2
 \end{smallmatrix}\bigr)$ we have (see \cite{shimura_halfintegral})
 \begin{equation*}
   f|_kU(p^2)=f|_k\Gamma_0(4N)\alpha_0\Gamma_0(4Np)=(p^2)^{\frac{k}{2}-1}\sum_{b=0}^{p^2-1}f|_k\alpha_b^*.
 \end{equation*}
Moreover, we have
$\Gamma_0(4N)\alpha_0\Gamma_0(4Np)=\cup_b\Gamma_0(4N)\alpha_b$, and by
Section 3.1 of \cite{shimura_book},
$\Gamma_0(4N)\alpha_0\Gamma_0(4Np)\Gamma_0(4Np)1_2\Gamma_0(4N)=(p+1)p^2\Gamma_0(4N)\alpha_0\Gamma_0(4N)$.
From this the first assertion follows, and the second one follows in
the same way as in the integral weight case, using Lemma
\ref{trace-skp}, which is valid for half integral weight too. 
\end{proof}
As mentioned in the introduction, because of the lack of a
satisfactory theory of oldforms and newforms in the half integral
weight case we finish the investigation of this case here without
trying to find good orthogonal bases for the space of all cusp forms.
 \section{Fourier coefficients of cusp forms}
For the rest of this paper we concentrate again on the case of modular
forms of integral weight $k$. 
 \begin{theorem}
The space $S_k(\Gamma_0(M),\chi)$ has an orthonormal basis $(h_1,\ldots,h_d)$, where each 
$h_i$ is an eigenform of all Hecke operators $T(p)$ for $p\nmid M$ and where the Fourier coefficients
$a(h_i,n)$ satisfy
 \begin{equation}\label{on_estimate}
|a(h_i,n)| \leq 2 \sqrt{\pi} e^{2\pi}\sigma_0(n) n^{\frac{k-1}{2}} \cdot M^{\frac{1}{2}}\cdot \prod_{p|M} \frac{(1+\frac{1}{p})^3}{\sqrt{1-\frac{1}{p^4}}}.
 \end{equation}
 \end{theorem}
 
 \begin{proof}
We write $g_j = \phi_{p,j}(\tilde{f})$ for the basis vectors $g_j \in W_p(f)$ constructed in Theorem
\ref{ogbasis_prime}, where $f\in S_k(\Gamma_0(M),\chi)$ is a newform
of level $N_f$ with  $p^jN\mid M$ and $\tilde{f}=\frac{f}{\sqrt{\langle
    f, f\rangle}}$. 
Obviously,
the operators $\phi_{p,j_p}$, defined for an arbitrary modular form by
the right hand sides of
(\ref{basisvectors1}), (\ref{basisvectors2}),  for distinct primes $p$ commute, and  as noticed
after Lemma \ref{product_decomposition-Lemma} the space
$S_k(\Gamma_0(M),\chi)$ 
has then an orthogonal basis consisting of the $(\prod_{p|M} \phi_{p,j_p})(\tilde{f})$, where $f$ runs over the primitive
forms of levels $N_f\mid M$ in $S_k(\Gamma_0(M),\chi)$ and $j_p \geq
0$ over the integers satisfying $N_fp^{j_p}\mid M$.
 \medskip

Examining the Proof of Theorem \ref{ogbasis_prime} we see that the  Petersson norm of $(\prod_{i} \phi_{p_i,j_{p_i}})
(\tilde{f})$ is equal to the product over $i$ of the norms of the $\phi_{p_i,j_{p_i}}(\tilde{f})$, which were computed in that 
theorem.
 \medskip

Analogously, we can decompose the computation of a bound for the Fourier coefficients of $(\prod_{i} \phi_{p_i,j_{p_i}})
(\tilde{f})$ into the computation of such a bound for each $\phi_{p_i,j_{p_i}}(\tilde{f})$. Looking at the $g_j$
again, we have for $p|N_f$ (using $|a(f,n)| \leq
\sigma_0(n)n^{\frac{(k-1)}{2}}$ and $\vert \lambda(1,p)\vert \le
p^{\frac{k-1}{2}}$ for primitive forms $f$ and $p\mid N_f$) 
 \begin{align*}
 \langle f,f \rangle^{\frac{1}{2}} |a(g_0,n)| \leq &\sigma_0(n) n^{\frac{k-1}{2}} \mbox{ and}\\
 \langle f,f \rangle^{\frac{1}{2}} |a(g_j,n)| \leq & p^{\frac{jk}{2}} \sigma_0(\frac{n}{p^j})(\frac{n}{p^j})^{\frac{k-1}{2}}\\
   &  +p^{\frac{jk}{2}}  p^{-\frac{(k+1)}{2}} \sigma_0(\frac{n}{p^{j-1}})(\frac{n}{p^{j-1}})^{\frac{k-1}{2}}
 \end{align*}
for $j \geq 1$, where the terms involving
$\frac{n}{p^j},\frac{n}{p^{j-1}}$ appear only if the respective
quotient is integral.  
This gives  $\langle f,f \rangle^{\frac{1}{2}} |a(g_j,n)| \leq
\sigma_0(n) n^{\frac{k-1}{2}} p^{\frac{j}{2}} (1+\frac{1}{p}) $ for
$j\ge 1$, and we see that this estimate holds indeed for all $j$.
\medskip

For $p\nmid N$ we obtain (with $|\lambda(1,p)| \leq
2p^{\frac{k-1}{2}}$ for $p\nmid N_f$):
 \begin{align*}
\langle f,f \rangle^{\frac{1}{2}} |a(g_0,n)| \leq & \sigma_0(n) n^{\frac{k-1}{2}}\\
\langle f,f \rangle^{\frac{1}{2}} |a(g_1,n)| \leq & p^{\frac{k}{2}} \sigma_0(\frac{n}{p})(\frac{n}{p})^{\frac{k-1}{2}}\\
  & + 2\sigma_0(n)n^{\frac{k-1}{2}} \cdot \frac{p^{\frac{k-1}{2}}}{p^{\frac{k}{2}}(1+\frac{1}{p})}\\
  \leq & \sigma_0(n) n^{\frac{k-1}{2}} p^{\frac{1}{2}}(1+\frac{2}{p(1+\frac{1}{p})})
 \end{align*}
and for $ j \geq 2$
 \begin{align*}
 \langle f,f\rangle^{\frac{1}{2}} |a(g_j,n)| \leq & p^{\frac{jk}{2}} (\sigma_0(\frac{n}{p^j})(\frac{n}{p^j})^{\frac{k-1}{2}} +
 2 \cdot \frac{p^{\frac{k-1}{2}}}{p^k} \cdot \sigma_0(\frac{n}{p^{j-1}})(\frac{n}{p^{j-1}})^{\frac{k-1}{2}}\\
 & +\frac{1}{p^{k+1}} \sigma_0( \frac{n}{p^{j-2}})(\frac{n}{p^{j-2}})^{\frac{k-1}{2}})\\
 \leq & \sigma_0(n) n^{\frac{k-1}{2}} p^{\frac{j}{2}} (1+\frac{1}{p})^2,
 \end{align*}
and we see that the latter bound holds for all $j$.
 \medskip

Finally, to estimate  $\langle f,f \rangle $ for the primitive form
$f$ from below we choose the fundamental domain
${\mathcal F}$ so that it contains $\{x+iy \in H\mid \vert x \vert
<\frac{1}{2}, y>1\}$, use $a(f,1)=1$ and get as in 
\cite{fomenko1} 
 \begin{equation*}
 \langle f,f \rangle \geq (  4\pi e^{4\pi}N_f \cdot \prod_{p|N_f} (1+\frac{1}{p}))^{-1} 
 \end{equation*}
from the trivial bound 
$\int_{\mathcal F}\vert f(x+iy)\vert^2 y^{k-2}dxdy\ge \int_1^\infty
\exp(-4\pi y)dy$.

Improvements on this are possible by \cite{Go-Ho-Li, Ho-Lo} but have been made effective so
far only in few cases, see \cite{rouse}. 
At least if the conductor $M_\chi$ of the character $\chi$ is small
compared to $M$
these don't give much for our present purpose because of the 
additional factors coming from oldforms which we computed above. 

\medskip
Putting things together and comparing the bounds in the cases $p\mid N$ and
$p \nmid N$ , we arrive for $h$ equal to the quotient of one of the
$\prod_{p|M} \phi_{p,j_p}(\tilde{f})$ by its Petersson norm at the
common bound
 \begin{equation*}
 |a(h,n)| \leq 2\sqrt{\pi} e^{2 \pi} \sigma_0(n) n^{\frac{k-1}{2}} M^{\frac{1}{2}} \prod_{p|M} \frac{(1+\frac{1}{p})^3}
 {\sqrt{1-\frac{1}{p^4}}}
 \end{equation*}
for both cases as asserted.
 \end{proof}

 \begin{theorem}\label{fourier_estimate}
Let $F\in S_k(\Gamma_0(M),\chi)$. Then the Fourier coefficients $a(F,n)$ satisfy
 \begin{equation}
 |a(F,n)| \leq 2\sqrt{\pi} e^{2 \pi}\sqrt{\langle F,F \rangle} \cdot (\dim S_k(\Gamma_0(M),\chi))^{\frac{1}{2}} \cdot \sigma_0(n) n^{\frac{k-1}{2}}
  M^{\frac{1}{2}} \cdot \prod_{p|M} \frac{(1+\frac{1}{p})^3}{\sqrt{1-\frac{1}{p^4}}}.
 \end{equation}
 \end{theorem}

 \begin{proof}
This follows immediately from the previous theorem, using the
Cauchy-Schwarz inequality:  We write $F=\sum_{\nu=1}^dc_\nu h_\nu$,
where $d=\dim S_k(\Gamma_0(M),\chi)$ and where $h_\nu$ runs through
the orthonormal basis from the previous theorem. We have then $\langle
F,F \rangle=\sum_{\nu=1}^d|c_\nu|^2$ and
\begin{eqnarray}
  \label{proof_maintheorem}
  |a(F,n)|&=&|\sum_{\nu=1}^dc_\nu a(h_\nu,n)|\nonumber\\
          &=& |(c_1,\ldots,c_d) \cdot (a(h_1,n),\ldots,a(h_d,n)|\nonumber\\
  &\le&  |(c_1,\ldots,c_d)||(a(h_1,n),\ldots,a(h_d,n))|\nonumber\\
  &\le&\sqrt{\langle F,F\rangle}\max(|a(h_\nu,n)|) \sqrt{d}.
\end{eqnarray}
Inserting the bound (\ref{on_estimate}) for $\max(|a(h_\nu,n)|$ from
the previous theorem
we obtain the assertion.
 \end{proof}

 \begin{remark}
   \begin{enumerate}
\item As indicated above it should be possible to improve on the factor
  $M^{\frac{1}{2}}$ in the bound for $a(F,n)$ if the conductor $M_\chi$ of the character $\chi$ is equal to
  $M$ or at least relatively
  large compared to $M$ by using an effective
    version of the bound for the Petersson norm of a primitive form  from  \cite{Go-Ho-Li, Ho-Lo} .
   \item 
For $\gcd(n,M) = 1$ we obtain the better estimate
 \begin{equation*}
 |a(h_i,n)| \leq 2\sqrt{\pi} e^{2 \pi}\sigma_0(n) n^{\frac{k-1}{2}} \cdot M^{\frac{1}{2}} \prod_{p|M} \frac{(1+\frac{1}{p})}
  {\sqrt{1-\frac{1}{p^4}}}
 \end{equation*}
in Theorem \ref{fourier_estimate} and hence
 \begin{equation*}
|a(F,n)| \leq 2\sqrt{\pi} e^{2 \pi}\sqrt{\langle F,F \rangle} \cdot (\dim S_k(\Gamma_0(M),\chi))^{\frac{1}{2}} \cdot \sigma_0(n) n^{\frac{k-1}{2}}
  M^{\frac{1}{2}} \cdot \prod_{p|M} \frac{(1+\frac{1}{p})}{\sqrt{1-\frac{1}{p^4}}}.
 \end{equation*}
   \end{enumerate}
 \end{remark}


 \bigskip

Rainer Schulze-Pillot, Abdullah Yenirce\\
Fachrichtung Mathematik,
Universit\"at des Saarlandes (Geb. E2.4)\\
Postfach 151150, 66041 Saarbr\"ucken, Germany\\
email: abdullahyenirce@googlemail.com,
schulzep@math.uni-sb.de
\end{document}